\documentclass[12pt,a4paper]{article}
\usepackage[dvipdfmx]{hyperref,graphicx,color}
\usepackage{amsmath,amsthm,amssymb,enumerate,mathrsfs}
\usepackage[numbers]{natbib}

\def\A{\mathcal{A}}

\def\C{\mathcal{C}}
\def\E{\mathcal{E}}

\def\K{{\mathcal{K}}}

\def\M{\mathcal{M}}
\def\N{\mathbb{N}}
\def\P{\mathcal{P}}
\def\R{\mathbb{R}}

\def\U{\mathcal{U}}
\def\V{{\mathcal{V}}}

\theoremstyle{plain}
\newtheorem{thm}{Theorem}[section]
\newtheorem{cor}{Corollary}[section]
\newtheorem{lem}{Lemma}[section]
\newtheorem{prop}{Proposition}[section]

\theoremstyle{definition}
\newtheorem{assmp}{Assumption}[section]

\newtheorem{dfn}{Definition}[section]

\newtheorem{rem}{Remark}[section]

\numberwithin{equation}{section}
\allowdisplaybreaks

\title{Relaxation and Purification for Nonconvex Variational Problems in Dual Banach Spaces: The Minimization Principle in Saturated Measure Spaces\thanks{An earlier version of this paper was presented at the ``10th IFAC Symposium on Nonlinear Control Systems'', Monterey and at the ``International Workshop on Nonlinear Analysis and Convex Analysis'', Research Institute for Mathematical Sciences, Kyoto University, respectively held in August 2016. I am grateful to Vladimir Kadets for his valuable remark, which led to the improvement of the Lyapunov convexity theorem presented in this paper, M. Ali Khan for candid discussions on the overall subject, and Sehie Park, Wataru Takahashi, Hong-Kun Xu, two anonymous referees, and the editor of the journal for their helpful comments. This research is supported by JSPS KAKENHI Grant No.\ 26380246 from the Ministry of Education, Culture, Sports, Science and Technology, Japan.}}
\date{\today}
\author{Nobusumi Sagara 
\\{\small Department of Economics, Hosei University} \\
{\small 4342, Aihara, Machida, Tokyo 194--0298, Japan} \\
{\small e-mail: nsagara@hosei.ac.jp}}

\begin{document}
\maketitle
\setcounter{page}{0}
\thispagestyle{empty}
\clearpage

\begin{abstract} 
We formulate bang-bang, purification, and minimization principles in dual Banach spaces with Gelfand integrals and provide a complete characterization of the saturation property of finite measure spaces. We also present an application of the relaxation technique to large economies with infinite-\hspace{0pt}dimensional commodity spaces, where the space of agents is modeled as a finite measure space. We propose a ``relaxation'' of large economies, which is regarded as a reasonable convexification of original economies. Under the saturation hypothesis, the relaxation and purification techniques enable us to prove the existence of Pareto optimal allocations without convexity assumptions. \\


\end{abstract}

\section{Introduction}
Recent advances in measure theory have stimulated the resurgence of Maharam's classical work \cite{ma42} to the formulation of the ``saturation'' of measure spaces. The notion of saturation is rooted in the study of Loeb spaces. A refinement of this notion appeared in \cite{fk02,fr12,hk84} and the current definition is formulated in \cite{ks09}. Saturated measure spaces possess an essentially uncountably generated $\sigma$-\hspace{0pt}algebra and saturation is a strengthened notion of nonatomicity. As a consequence, saturation remedies a well-known failure of the Lyapunov convexity theorem in infinite-dimensional Banach spaces and guarantees its validity, even in nonseparable locally convex spaces; see \cite{gp13,ks13,ks15,ks16a}. Earlier remedies for the failure of the Lyapunov convexity theorem in infinite dimensions were given in \cite{su92,su97} for the case with Banach spaces with the Radon--Nikodym property in the setting of nonatomic Loeb spaces. The recovery of the Lyapunov convexity theorem in infinite dimensions is undoubtedly useful,  especially in variational analysis. 

It is noteworthy that saturation is not only sufficient, but also necessary for the Lyapunov convexity theorem in separable Banach spaces, as shown in \cite{ks13,ks15,ks16a}. Furthermore, it is also necessary and sufficient for the convexity of the Bochner and the Gelfand integrals of a multifunction taking values in separable Banach spaces and their dual spaces (see \cite{po08,sy08}), and for Fatou's lemma in the Bochner and the Gelfand integral settings; see \cite{ks14a,kss16}. In this sense, saturation is the best possible structure on measure spaces for dealing with vector measures taking values in infinite-dimensional vector spaces. Based on these findings, in \cite{ks14b} the interplay of the Lyapunov convexity theorem and the bang-bang principle in separable Banach spaces were established with the Bochner integral setting and the relaxation and the purification techniques for nonconvex variational problems were given a full-fledged treatment under the saturation hypothesis. See also \cite{ls06,ls09,po09} for earlier results on the purification principle. 

This work details a further step toward the equivalence results on saturation along the lines of the aforementioned literature. The purpose of the paper is twofold. First, we formulate the bang-bang and purification principles in dual spaces of a separable Banach space with Gelfand integrals and provide a complete characterization of the saturation property of finite measure spaces. To this end, we make the best use of ``relaxed controls'' developed in \cite{mc67,wa72,yo69}. In particular, we provide the equivalence of saturation and the existence of solutions to nonconvex variational problems with Gelfand integrals constraints. This is a novel aspect not pursued in the author's previous work \cite{ks14b}, which is referred to the ``minimization principle'' for saturation. 

Second, we present an application of the relaxation technique to large economies with infinite-dimensional commodity spaces, where the space of agents is modeled as a finite measure space following \cite{au64,au66}. We propose a ``relaxation'' of large economies along the lines of \cite{ba08,ks16b}, which is regarded as a reasonable convexification of original economies. We introduce the notion of relaxed Pareto optimality and derive the existence of Pareto optimal allocations of the original economy under the saturation hypothesis. The relaxation and purification techniques enable us to prove the existence of Pareto optimal allocations without convexity assumptions.   

In the following section, we provide a brief overview of Gelfand integrals for functions and multifunctions taking values in dual spaces of a Banach space and derive the compactness property of the set of Gelfand integrable selectors of a multifunction. Thereafter, the organization of the paper proceeds to address the two purposes described.

\section{Preliminaries}
\subsection{Gelfand Integrals of Functions}
Let $(T,\Sigma,\mu)$ be a finite measure space and $E$ be a real Banach space with the dual system $\langle E,E^* \rangle$, where $E^*$ is the norm dual of $E$. A function $f:T\to E^*$ is \textit{weakly$^*\!$ scalarly measurable} if the scalar function $\langle f(\cdot),x \rangle$ on $T$ is measurable for every $x\in E$. We say that weakly$^*\!$ scalarly measurable functions $f$ and $g$ are \textit{weakly$^*\!$ scalarly equivalent} if $\langle f(t)-g(t),x \rangle=0$ for every $x\in E$ a.e.\ $t\in T$ (the exceptional $\mu$-\hspace{0pt}null set depending on $x$). Denote by $\mathrm{Borel}(E^*,\mathit{w}^*)$ the Borel $\sigma$-\hspace{0pt}algebra of $E^*$ generated by the weak$^*\!$ topology. If $E$ is a separable Banach space, then $E^*$ is separable with respect to the weak$^*\!$ topology (see \cite[Lemma I.3.4 of Part II]{sc73}) and it is a locally convex Suslin space under the weak$^*\!$ topology; see \cite[p.\,67]{th75}. Hence, under the separability of $E$, a function $f:T\to E^*$ is weakly$^*\!$ scalarly measurable if and only if it is measurable with respect to $\mathrm{Borel}(E^*,\mathit{w}^*)$; see \cite[Theorem 1]{th75}.

We say that a weakly$^*\!$ scalarly measurable function $f:T\to E^*$ is \textit{weakly$^*\!$ scalarly integrable} if the scalar function $\langle f(\cdot),x \rangle$ is integrable for every $x\in E$. A weakly$^*\!$ scalarly measurable function $f$ is \textit{Gelfand integrable} over $A\in \Sigma$ if there exists $x^*_A\in E^*$ such that 
$$
\langle x^*_A,x \rangle=\int_A\langle f(t),x \rangle d\mu \quad\text{for every $x\in E$}.
$$ 
The element $x^*_A$ is called the \textit{Gelfand integral} (or the \textit{weak$^*\!$ integral}) of $f$ over $A$ and denoted by $\int_Afd\mu$. Every weakly$^*\!$ scalarly integrable function is weakly$^*\!$ integrable; see \cite[Theorem 11.52]{ab06}. Denote by $G^1(\mu,E^*)$ (abbreviated to $G^1_{E^*}$) the equivalence classes of Gelfand integrable functions with respect to weakly$^*\!$ scalarly equivalence, normed by
$$
\| f \|_{\mathit{G}^1}=\sup_{x\in B_E}\int_T|\langle f(t),x \rangle |d\mu,
$$
where $B_E$ is the closed unit ball in $E$. This norm is called the \textit{Gelfand norm}, whereas the normed space $(G^1(\mu,E^*), \|\cdot \|_{\mathit{G}^1})$, in general, is not complete.

\subsection{The Topology of Pointwise Convergence on $G^1_{E^*}$}
Let $L^\infty(\mu)$ be the space of $\mu$-essentially bounded measurable functions on $T$ with the essential sup norm. Denote by $B(L^\infty(\mu)\times E)$ the space of bilinear forms on the product space $L^\infty(\mu)\times E$. For each pair $(\varphi,x)\in L^\infty(\mu)\times E$, the linear functional $\varphi\otimes x$ on $B(L^\infty(\mu)\times E)$ is defined by $(\varphi\otimes x)(M)=M(\varphi,x)$, $M\in B(L^\infty(\mu)\times E)$. The \textit{tensor product} $L^\infty(\mu)\otimes E$ of $L^\infty(\mu)$ and $E$ is the subspace of the algebraic dual of $B(L^\infty(\mu)\times E)$ spanned by these elements $\varphi\otimes x$. Thus, a typical \textit{tensor} $f^*$ in $L^\infty(\mu)\otimes E$ has a (not necessarily unique) representation $f^*=\sum_{i=1}^n\varphi_i\otimes x_i$ with $\varphi_i\in L^\infty(\mu)$, $x_i\in E$, $i=1,\dots,n$. A bilinear form on $G^1(\mu,E^*)\times (L^\infty(\mu)\otimes E)$ is given by
$$
\langle f^*,f \rangle=\sum_{i=1}^n\int_T\varphi_i(t)\langle f(t),x_i \rangle d\mu=\sum_{i=1}^n\left\langle \int_T\varphi_i(t)f(t)d\mu,x_i \right\rangle
$$
for $f\in G^1(\mu,E^*)$ and $f^*=\sum_{i=1}^n\varphi_i\otimes x_i\in L^\infty(\mu)\otimes E$, where $\int\varphi_ifd\mu$ denotes the Gelfand integral of $\varphi_if\in G^1(\mu,E^*)$. The pair of these spaces $\langle G^1(\mu,E^*),L^\infty(\mu)\otimes E \rangle$ equipped with this bilinear form is a dual system. 

The coarsest topology on $G^1(\mu,E^*)$ such that the linear functional $f\mapsto \langle f^*,f \rangle$ is continuous for every $f^*\in L^\infty(\mu)\otimes E$, denoted by $\sigma(G^1_{E^*},L^\infty\otimes E)$, is the \textit{topology of pointwise convergence} on $L^\infty(\mu)\otimes E$, generated by the family of seminorms $\{ p_{f^*}\mid f^*\in L^\infty\otimes E \}$, where $p_{f^*}(f)=|\langle f^*,f \rangle|$, $f\in G^1(\mu,E^*)$. Thus, $G^1(\mu,E^*)$ endowed with the $\sigma(G^1_{E^*},L^\infty\otimes E)$ is a locally convex space. Let $L^1(\mu)$ be the space of $\mu$-integrable functions with the $L^1$-norm. A net $\{ f_\alpha \}$ in $G^1(\mu,E^*)$ converges to $f\in G^1(\mu,E^*)$ for the $\sigma(G^1_{E^*},L^\infty\otimes E)$-\hspace{0pt}topology if and only if for every $x\in E$ the net $\{ \langle f_\alpha(\cdot),x \rangle \}$ in $L^1(\mu)$ converges weakly to $\langle f(\cdot),x \rangle \in L^1(\mu)$. It is evident that the $\sigma(G^1_{E^*},L^\infty\otimes E)$-\hspace{0pt}topology is coarser than the weak topology $\sigma(G^1_{E^*},(G^1_{E^*})^*)$.

\subsection{Gelfand Integrals of Multifunctions}
\label{subsec}
Let $\Gamma:T\twoheadrightarrow E^*$ be a multifunction. (By \textit{multifunction} we always mean a set-valued mapping with nonempty values.) Denote by $\overline{\mathrm{co}}^{\mathit{\,w}^*}\Gamma:T\twoheadrightarrow E^*$ the multifunction defined by the weakly$^*\!$ closed convex hull of $\Gamma(t)$ and by $\mathrm{ex}\,\overline{\mathrm{co}}^{\,\mathit{w}^*}\Gamma:T\twoheadrightarrow E^*$ the multifunction defined by the extreme points of $\overline{\mathrm{co}}^{\mathit{\,w}^*}\Gamma(t)$. Let $s:(\cdot, C):E\to \R\cup \{+\infty \}$ be the \textit{support function} of a set $C\subset E^*$ defined by $s(x, C)=\sup_{x^*\in C}\langle x^*,x \rangle$. A multifunction $\Gamma$ is \textit{weakly$^*\!$ scalarly measurable} if the scalar function $s(x,\Gamma):T\to \R\cup\{ +\infty \}$ is measurable for every $x\in E$; it is \textit{weakly$^*\!$ scalarly integrable} if $s(x, \Gamma)$ is integrable for every $x\in E$. It follows from $s(x,\Gamma)=s(x,\overline{\mathrm{co}}^{\mathit{\,w}^*}\Gamma)=s(x,\mathrm{ex}\,\overline{\mathrm{co}}^{\,\mathit{w}^*}\Gamma)$ for every $x\in E$ that $\Gamma$ is weakly$^*\!$ scalarly measurable (resp.\ weakly$^*\!$ scalarly integrable) if and only if so is $\overline{\mathrm{co}}^{\mathit{\,w}^*}\Gamma$ (resp.\ $\mathrm{ex}\,\overline{\mathrm{co}}^{\,\mathit{w}^*}\Gamma$). The \textit{graph} of $\Gamma$ is defined by the set $\mathrm{gph}\,\Gamma=\{ (t,x^*)\in T\times E^*\mid x^*\in \Gamma(t) \}$. A multifunction $\Gamma$ is \textit{integrably bounded} if there exists $\varphi\in L^1(\mu)$ such that $\sup_{x^*\in \Gamma(t)}\| x^* \|\le \varphi(t)$ for every $t\in T$. If $\Gamma$ has weakly$^*\!$ compact convex values, then $\Gamma$ is weakly$^*\!$ scalarly measurable if and only if $\mathrm{gph}\,\Gamma\in \Sigma\otimes \mathrm{Borel}(E^*,\mathit{w}^*)$ whenever $(T,\Sigma,\mu)$ is complete and $E$ is separable; see \cite[Theorem III.30]{cv77}. 

A function $f:T\to E^*$ is a \textit{selector} of a multifunction $\Gamma$ if $f(t)\in \Gamma(t)$ a.e.\ $t\in T$. Denote by $\mathcal{S}^1_\Gamma$ the set of Gelfand integrable selectors of $\Gamma$. If $E$ is separable, then $E^*$ is Suslin with respect to the weak$^*\!$ topology, and hence, any multifunction $\Gamma$ with $\mathrm{gph}\,\Gamma\in \Sigma\otimes \mathrm{Borel}(E^*,\mathit{w}^*)$ admits a measurable selector whenever $(T,\Sigma,\mu)$ is complete; see \cite[Theorem III.22]{cv77}. Since any measurable selector from an integrably bounded multifunction $\Gamma$ is Gelfand integrable, $\mathcal{S}^1_\Gamma$ is nonempty for every integrably bounded multifunction $\Gamma$ with measurable graph whenever $E$ is separable. The Gelfand integral of $\Gamma$ is defined by
$$
\int_T\Gamma(t)d\mu=\left\{ \int_Tf(t)d\mu \mid f\in \mathcal{S}^1_\Gamma \right\}. 
$$

For the later use, we need the following result. 

\begin{lem}
\label{lem1}
Let $(T,\Sigma,\mu)$ be a complete finite measure space, $E$ be a separable Banach space, and $\Gamma:T\twoheadrightarrow E^*$ be an integrably bounded, weakly$^*\!$ closed convex-\hspace{0pt}valued multifunction with $\mathrm{gph}\,\Gamma\in \Sigma\otimes \mathrm{Borel}(E^*,\mathit{w}^*)$. Then $\mathcal{S}^{1}_\Gamma$ is nonempty, $\sigma(G^1_{E^*},L^\infty\otimes E)$-\hspace{0pt}compact, and convex. 
\end{lem}

\begin{proof}
The nonemptiness and convexity of $\mathcal{S}^1_\Gamma$ are obvious. Let $L^1(\mu)^E$ be the space of functions from $E$ to $L^1(\mu)$ endowed with the product topology $\tau_p$ induced by the weak topology of $L^1(\mu)$. This means that the convergence of a net $\{ v_\alpha \}$ in $L^1(\mu)^E$ is characterized by the pointwise convergence: $v_\alpha\to v$ in the $\tau_p$-\hspace{0pt}topology if and only if $v_\alpha(x) \to v(x)$ weakly in $L^1(\mu)$ for every $x\in E$. For each $f\in G^1(\mu,E^*)$ define $v_f\in L^1(\mu)^E$ by $v_f(x)=\langle f(\cdot),x \rangle$. Then the convex set $\V=\{ v_f\mid f\in \mathcal{S}^1_\Gamma \}$ is $\tau_p$-\hspace{0pt}compact in $L^1(\mu)^E$; see \cite[Proposition 2.3, Theorem 4.5, and Condition $(\alpha)$ of p.\,885]{ckr11}. Define $\Psi:\V\to G^1(\mu,E^*)$ by $\Psi v_f=f$. (Since $f\mapsto v_f$ is one-to-one, $\Psi$ is well-defined.) To demonstrate the continuity of $\Psi$, let $v_{f_\alpha}\to v_f$ in $\V$ for the $\tau_p$-\hspace{0pt}topology. Since $\langle f_\alpha(\cdot),x \rangle=v_{f_\alpha}(x)\to v_f(x)=\langle f(\cdot),x \rangle$ weakly in $L^1(\mu)$ for every $x\in E$, we have $\int \varphi(t)\langle (\Psi v_{f_\alpha})(t),x \rangle d\mu\to \int \varphi(t)\langle (\Psi v_f)(t),x \rangle d\mu$ for every $(\varphi,x)\in L^\infty(\mu)\times E$. Thus, $\Psi$ is continuous in the $\tau_p$-\hspace{0pt}topology in $L^1(\mu)^E$ and  the $\sigma(G^1_{E^*},L^\infty\otimes E)$-\hspace{0pt}topology in $G^1(\mu,E^*)$. Hence, $\Psi(\V)=\mathcal{S}^1_\Gamma$ is $\sigma(G^1_{E^*},L^\infty\otimes E)$-\hspace{0pt}compact. 
\end{proof}

\section{Relaxation and Purification in Saturated Measure Spaces}
\subsection{The Bang-Bang Principle}
In the sequel, we always assume the completeness of $(T,\Sigma,\mu)$. A finite measure space $(T,\Sigma,\mu)$ is said to be \textit{essentially countably generated} if its $\sigma$-\hspace{0pt}algebra can be generated by a countable number of subsets together with the null sets; $(T,\Sigma,\mu)$ is said to be \textit{essentially uncountably generated} whenever it is not essentially countably generated. Let $\Sigma_S=\{ A\cap S\mid A\in \Sigma \}$ be the $\sigma$-\hspace{0pt}algebra restricted to $S\in \Sigma$. Denote by $L^1_S(\mu)$ the space of $\mu$-integrable functions on the measurable space $(S,\Sigma_S)$ whose elements are restrictions of functions in $L^1(\mu)$ to $S$. An equivalence relation $\sim$ on $\Sigma$ is given by $A\sim B \Leftrightarrow \mu(A\triangle B)=0$, where $A\triangle B$ is the symmetric difference of $A$ and $B$ in $\Sigma$. The collection of equivalence classes is denoted by $\Sigma(\mu)=\Sigma/\sim$ and its generic element $\widehat{A}$ is the equivalence class of $A\in \Sigma$. We define the metric $\rho$ on $\Sigma(\mu)$ by $\rho(\widehat{A},\widehat{B})=\mu(A\triangle B)$. Then $(\Sigma(\mu),\rho)$ is a complete metric space (see \cite[Lemma 13.13]{ab06}) and $(\Sigma(\mu),\rho)$ is separable if and only if $L^1(\mu)$ is separable (see \cite[Lemma 13.14]{ab06}). The \textit{density} of $(\Sigma(\mu),\rho)$ is the smallest cardinal number of the form $|\U|$, where $\U$ is a dense subset of $\Sigma(\mu)$. 

\begin{dfn}
A finite measure space $(T,\Sigma,\mu)$ is \textit{saturated} if $L^1_S(\mu)$ is nonseparable for every $S\in \Sigma$ with $\mu(S)>0$. We say that a finite measure space has the \textit{saturation property} if it is saturated.
\end{dfn}

Saturation implies nonatomicity and several equivalent definitions for saturation are known; see \cite{fk02,fr12,hk84,ks09}. One of the simple characterizations of the saturation property is as follows. A finite measure space $(T,\Sigma,\mu)$ is saturated if and only if $(S,\Sigma_S,\mu)$ is essentially uncountably generated for every $S\in \Sigma$ with $\mu(S)>0$. The saturation of finite measure spaces is also synonymous with the uncountability of the density of $\Sigma_S(\mu)$ for every $S\in \Sigma$ with $\mu(S)>0$; see \cite[331Y(e)]{fr12}. An inceptive notion of saturation already appeared in \cite{ka44,ma42}. The significance of the saturation property lies in the fact that it is necessary and sufficient for the weak$^*\!$ compactness and the convexity of the Gelfand integral of a multifunction as well as the Lyapunov convexity theorem; see \cite{ks13,ks15,ks16a,po08,sy08}. We present here the relevant result from \cite[Theorem 3.3 and 3.6]{ks15} with a slight extension for later use. 

\begin{prop}[Lyapunov convexity theorem]
\label{lyap1}
Let $(T,\Sigma,\mu)$ be a finite measure space and $E$ be a sequentially complete, separable, locally convex space. If $(T,\Sigma,\mu)$ is saturated, then for every $\mu$-continuous vector measure $m:\Sigma\to E$, its range $m(\Sigma)$ is weakly compact and convex. Conversely, if every $\mu$-continuous vector measure $m:\Sigma\to E$ has the weakly compact convex range, then $(T,\Sigma,\mu)$ is saturated whenever $E$ is an infinite-\hspace{0pt}dimensional locally convex space such that there is an infinite-dimensional Banach space $X$ that admits an injective continuous linear operator from $X$ into $E$. 
\end{prop}

\begin{proof}
We show only the converse implication. If $(T,\Sigma,\mu)$ is not saturated, then for every infinite-dimensional Banach space $X$ there is a $\mu$-continuous vector measure $n:\Sigma\to X$ such that its range $n(\Sigma)$ is not weakly compact, convex subset of $X$; see \cite[Lemma 4.1 and Theorem 4.2]{ks13}. Let $\Phi:X\to E$ be an injective continuous linear operator and define the $\mu$-continuous vector measure $m:\Sigma\to E$ by $m=\Phi\circ n$. By construction, $m$ does not possess the weakly compact convex range in $E$. 
\end{proof}

\begin{cor}
\label{lyap2}
Let $(T,\Sigma,\mu)$ be a finite measure space and $E$ be a separable Banach space. If $(T,\Sigma,\mu)$ is saturated, then for every $\mu$-continuous vector measure $m:\Sigma\to E^*$, its range $m(\Sigma)$ is weakly$^*\!$ compact and convex. Conversely, if every $\mu$-continuous vector measure $m:\Sigma\to E^*$ has the weakly$^*\!$ compact convex range, then $(T,\Sigma,\mu)$ is saturated whenever $E$ is infinite dimensional. 
\end{cor}

\begin{proof}
Since $E^*$ is sequentially complete with respect to the weak$^*\!$ topology consistent with duality (see \cite[Corollary 2.6.21]{me98}), the range $m(\Sigma)$ is weakly$^*\!$ compact and convex by Proposition \ref{lyap1}. For the converse implication, let $\imath_{E^*}: (E^*,\| \cdot \|)\to (E^*,\mathit{w}^*)$ be the identity map on $E^*$, which is an injective continuous linear operator. Therefore, $E^*$ satisfies the hypothesis in Proposition \ref{lyap1}. 
\end{proof}

The following result is the \textit{bang-bang principle} in infinite dimensions, an analogue of \cite[Theorem 4.1]{ks14b} in the dual space setting. 

\begin{thm}[bang-bang principle]
\label{BBP}
Let $(T,\Sigma,\mu)$ be a saturated finite measure space, $E$ be a separable Banach space, and $\Gamma:T\twoheadrightarrow E^*$ be an integrably bounded, weakly$^*\!$ closed-valued multifunction with $\mathrm{gph}\,\Gamma\in \Sigma\otimes \mathrm{Borel}(E^*,\textit{w}^*)$. Then 
\begin{equation}
\label{bbp}
\int_T\Gamma(t)d\mu=\int_T\mathrm{ex}\,\overline{\mathrm{co}}^{\,\mathit{w}^*}\Gamma(t)d\mu. \tag{BBP}
\end{equation}
\end{thm}

\begin{proof}
The saturation property guarantees that $\int\Gamma d\mu$ is weakly$^*\!$ compact and convex with equality $\int\Gamma d\mu=\int \overline{\mathrm{co}}^{\,\mathit{w}^*}\Gamma d\mu$ by \cite[Theorem 4]{po08} and \cite[Proposition 1]{sy08}. It thus suffices to show that $\int \overline{\mathrm{co}}^{\,\mathit{w}^*}\Gamma d\mu=\int \mathrm{ex}\,\overline{\mathrm{co}}^{\,\mathit{w}^*}\Gamma d\mu$. Define the integration operator $I:G^1(\mu,E^*)\to E^*$ by $I(f)=\int fd\mu$. It is easy to see that $I$ is continuous in the $\sigma(G^1_{E^*},L^\infty\otimes E)$- and the weak$^*\!$ topologies. Take a point $x^*\in I(\mathcal{S}_{\overline{\mathrm{co}}^{\,\mathit{w}^*}\Gamma}^1)=\int \overline{\mathrm{co}}^{\,\mathit{w}^*}\Gamma d\mu$ arbitrarily. Since $\overline{\mathrm{co}}^{\,\mathit{w}^*}\Gamma$ is integrably bounded and weakly$^*\!$ scalarly measurable, by Lemma \ref{lem1}, the set $I^{-1}(x^*)\cap \mathcal{S}_{\overline{\mathrm{co}}^{\,\mathit{w}^*}\Gamma}^1$ is a $\sigma(G^1_{E^*},L^\infty\otimes E)$-\hspace{0pt}compact, convex subset of $G^1(\mu,E^*)$, and hence, it has an extreme point $\hat{f}$ in view of the Krein--Milman theorem; see \cite[Corollary 7.68]{ab06}. It suffices to show that $\hat{f}\in \mathcal{S}_{\mathrm{ex}\,\overline{\mathrm{co}}^{\,\mathit{w}^*}\Gamma}^1$. Since $\hat{f}$ is a measurable selector of $\overline{\mathrm{co}}^{\,\mathit{w}^*}\Gamma$, the set $\{ t\in T\mid f(t)\in \overline{\mathrm{co}}^{\,\mathit{w}^*}\Gamma(t)\setminus \mathrm{ex}\,\overline{\mathrm{co}}^{\,\mathit{w}^*}\Gamma(t) \}$ is measurable; see \cite[p.\,108]{cv77}. Hence, if $\hat{f}\not\in \mathcal{S}_{\mathrm{ex}\,\overline{\mathrm{co}}^{\,\mathit{w}^*}\Gamma}^1$, then there exists $S\in \Sigma$ with $\mu(S)>0$ such that $\hat{f}(t)\not\in \mathrm{ex}\,\overline{\mathrm{co}}^{\,\mathit{w}^*}\Gamma(t)$ for every $t\in S$. By the integrable boundedness of $\overline{\mathrm{co}}^{\,\mathit{w}^*}\Gamma$ and \cite[Theorem IV.14]{cv77}, there exists $g\in G^1(\mu,E^*)$ such that $\hat{f}\pm g\in \mathcal{S}_{\overline{\mathrm{co}}^{\,\mathit{w}^*}\Gamma}^1$ and $g\ne 0$ on $S$. Define the vector measure $m:\Sigma\to E^*$ by $m(A)=\int_Agd\mu$ for $A\in \Sigma$. By Corollary \ref{lyap2}, the range of $m$ is weakly$^*\!$ compact and convex in $E^*$. Take $B\subset S$ with $m(B)=\frac{1}{2}m(S)$ and define $\hat{g}:T\to E^*$ by
$$
\hat{g}(t)=
\begin{cases}
\hspace{0.35cm} g(t) & \text{if $t\in S\setminus B$}, \\
-g(t) & \text{if $t\in B$}, \\
\quad 0 & \text{otherwise}.
\end{cases}
$$
Then $\hat{f}=\frac{1}{2}(\hat{f}+\hat{g})+\frac{1}{2}(\hat{f}-\hat{g})$ and $\hat{f}\pm \hat{g}\in \mathcal{S}_{\overline{\mathrm{co}}^{\,\mathit{w}^*}\Gamma}^1$. A simple calculation yields $\int \hat{g}dm=0$, and hence, $\hat{f}\pm \hat{g}\in I^{-1}(x^*)\cap \mathcal{S}_{\overline{\mathrm{co}}^{\,\mathit{w}^*}\Gamma}^1$. This contradicts the fact that $\hat{f}$ is an extreme point of $I^{-1}(x^*)\cap \mathcal{S}_{\overline{\mathrm{co}}^{\,\mathit{w}^*}\Gamma}^1$. 
\end{proof}

Theorem \ref{BBP} means that the Gelfand integral of any Gelfand integrable selector $f$ from $\Gamma$ is realized as that of some Gelfand integrable selector $g$ from $\mathrm{ex}\,\overline{\mathrm{co}}^{\,\mathit{w}^*}\Gamma$ in the sense that $\int fd\mu=\int gd\mu$. The saturation property of finite measure spaces guarantees the bang-bang principle for every integrably bounded, weakly$^*\!$ closed valued multifunction with measurable graph whenever $E$ is separable. 

Furthermore, the converse of Theorem \ref{BBP} is also true.  

\begin{thm}
\label{nec1}
Let $(T,\Sigma,\mu)$ be a finite measure space and $E$ be an infinite-\hspace{0pt}dimensional separable Banach space. If \eqref{bbp} holds for every integrably bounded, weakly$^*\!$ closed-valued multifunction $\Gamma:T\twoheadrightarrow E^*$ with $\mathrm{gph}\,\Gamma\in \Sigma\otimes \mathrm{Borel}(E^*,\textit{w}^*)$, then $(T,\Sigma,\mu)$ is saturated.
\end{thm}

\begin{proof}
Suppose, to the contrary, that $(T,\Sigma,\mu)$ is not saturated. Then there exists a Bochner (and hence Gelfand) integrable function $f\in G^1(\mu,E^*)$ such that the range of the indefinite Bochner integral $R_f=\{ \int_Afd\mu\in E^*\mid A\in \Sigma \}$ is not convex; see \cite[Lemma 4]{po08} or \cite[Remark 1(2)]{sy08}. Since the essential range of Bochner integrable functions is separable (see \cite[Theorem II.1.2]{du77}), we may assume that $f$ takes values in a separable subspace $V$ of $E^*$. Let $\Gamma_f:T\twoheadrightarrow V$ be a multifunction defined by $\Gamma_f(t)=\overline{\mathrm{co}}\{ 0,f(t) \}$, where the closed convex hull is taken with respect to the dual norm. Then $\Gamma_f$ is an integrably bounded, weakly$^*\!$ compact, convex-valued multifunction with $\mathrm{gph}\,\Gamma_f\in \Sigma\otimes \mathrm{Borel}(E^*,\textit{w}^*)$. Since $\mathcal{S}^1_{\Gamma_f}$ is convex in $G^1(\mu,E^*)$, the Gelfand integral $\int \Gamma_f d\mu$ is convex in $E^*$. Moreover, since the Gelfand integrable selectors from $\Gamma_f$ precisely coincide with the Bochner integrable selectors, we have $\mathrm{ex}\,\Gamma_f(t)=\{ 0,f(t) \}$ and $\mathcal{S}^1_{\mathrm{ex}\,\Gamma_f}=\{ f\chi_A \mid A\in \Sigma \}$. Since $\overline{\mathrm{co}}\{ 0,f(t) \}=\overline{\mathrm{co}}^{\,\mathit{w}^*}\{ 0,f(t) \}$ (see \cite[Theorem 5.98]{ab06}), if $\Gamma_f$ satisfies \eqref{bbp}, then $\int \Gamma_f d\mu=\int \mathrm{ex}\,\overline{\mathrm{co}}^{\,\mathit{w}^*}\Gamma_fd\mu=\int \mathrm{ex}\,\overline{\mathrm{co}}\,\Gamma_fd\mu=R_f$, an obvious contradiction to the convexity of $\int \Gamma_f d\mu$.
\end{proof}

\begin{rem}
The equivalence of saturation and the Lyapunov convexity theorem was established in \cite{ks13} for the case with separable Banach spaces and in \cite{gp13} for the case with dual spaces of a separable Banach space. Proposition \ref{lyap1} covers the abovementioned results and improves \cite{ks15}. The equivalence of saturation and the ``convexity principle'' in the sense that $\int\Gamma d\mu=\int\overline{\mathrm{co}}^{\,\mathit{w}^*}\Gamma d\mu$ for every integrably bounded, weakly$^*\!$ closed-valued multifunction $\Gamma$ with measurable graph was established in \cite{po08,sy08}; see also the earlier work by \cite{su97} in the setting of nonatomic Loeb spaces (which form a special class of saturated measure spaces). Theorems \ref{BBP} and \ref{nec1} are another characterization of saturation in terms of \eqref{bbp}. See \cite{ks14b} for a characterization of saturation in terms of the bang-bang principle with Bochner integrals in separable Banach spaces and \cite{ks16a} for that in terms of the bang-bang principle with Bourbaki--Kluv\'anek--Lewis integrals in separable locally convex spaces.  
\end{rem}

\subsection{The Purification Principle}
We denote by $\Pi(X)$ the set of probability measures on a Hausdorff topological space $X$ furnished with the Borel $\sigma$-algebra $\mathrm{Borel}(X)$. We endow $\Pi(X)$ with the \textit{topology of weak convergence} of probability measures (also called the \textit{narrow topology}), which is the coarsest topology on $\Pi(X)$ for which the integral functional $P\mapsto \int u dP$ on $\Pi(X)$ is continuous for every bounded continuous function $u:X\to \R$. Then $\Pi(X)$ is a Suslin space if and only if $X$ is a Suslin space; see \cite[Theorem 2.7 of Appendix in Part II]{sc73}. If $X$ is a Polish space, then the Borel $\sigma$-algebra on $\Pi(X)$ is the smallest $\sigma$-algebra for which the real-valued function $P\mapsto P(A)$ on $\Pi(X)$ is measurable for every $A\in \Sigma$; see \cite[Theorem 7.25]{bs78}.  

By $\M(T,X)$ we denote the space of measurable functions from $T$ to $X$ and by $\mathcal{R}(T,X)$ the space of measurable functions from $T$ to $\Pi(X)$. Each element in $\M(T,X)$ is called a \textit{control} and that in $\mathcal{R}(T,X)$ is called a \textit{relaxed control} (a \textit{Young measure}, a \textit{stochastic kernel}, or a \textit{transition probability}), which is a probability measure-valued control. If $X$ is a Polish space, then for every function $\lambda:T\to \Pi(X)$, the real-\hspace{0pt}valued function $t\mapsto \lambda(t)(C)$ is measurable for every $C\in \mathrm{Borel}(X)$ if and only if $\lambda$ is measurable. By $\Delta(X)$, we denote the set of Dirac measures on $X$, i.e., $\delta_x\in \Delta(X)$ whenever for every $C\in \mathrm{Borel}(X)$: $\delta_x(C)=1$ if $x\in C$ and $\delta_x(C)=0$ otherwise. If $X$ is a Polish space, then each control $f\in \M(T,X)$ is identified with the Dirac measure-valued control $\delta_{f(\cdot)}\in \mathcal{R}(T,X)$ satisfying $\delta_{f(t)}\in \Delta(X)$ for every $t\in T$. Given a multifunction $U:T\twoheadrightarrow X$, we say that $\lambda\in \mathcal{R}(T,X)$ is \textit{concentrated} on $U$ if $\lambda(t)(U(t))=1$ a.e.\ $t\in T$. 

We say that a function $\Phi:T\times X\to E^*$ is \textit{integrably bounded} if the multifunction $t\mapsto \Phi(t,X)\subset E^*$ is integrably bounded, i.e., there exists $\varphi\in L^1(\mu)$ such that $\| \Phi(t,x) \|\le \varphi(t)$ for every $(t,x)\in T\times X$. Recall that the real-valued function $\varphi:T\times X\to \R$ is a \textit{Carath\'eodory function} if $t\mapsto \varphi(t,x)$ is measurable for every $x\in X$ and $x\mapsto \varphi(t,x)$ is continuous for  every $t\in T$. If $X$ is a Polish space, then the Carath\'eodory function $\varphi$ is jointly measurable; see \cite[Lemma 4.51]{ab06}.

The proof of the following result is in \cite[Lemma 2.1]{ks14b}. 

\begin{lem}
\label{lem2}
Let $(T,\Sigma,\mu)$ be a probability space and $C$ be a weakly$^*\!$ closed convex subset of the dual space $E^*$ of a Banach space $E$. If $f:T\to E^*$ is a Gelfand integrable function with $f(T)\subset C$, then $\int fd\mu\in C$.   
\end{lem}

The next result is the \textit{purification principle} in infinite dimensions, an analogue of \cite[Theorem 5.1]{ks14b} in the dual space setting. 

\begin{thm}[purification principle]
\label{PP1}
Let $(T,\Sigma,\mu)$ be a saturated finite measure space, $E$ be a separable Banach space, and $X$ be a Suslin space. If $\Phi:T\times X\to E^*$ is an integrably bounded measurable function such that $\Phi(t,\cdot):X\to E^*$ is continuous in the weak$^*\!$ topology of $E^*$ for every $t\in T$ and $U:T\twoheadrightarrow X$ is a compact-\hspace{0pt}valued multifunction with $\mathrm{gph}\,U\in \Sigma\otimes \mathrm{Borel}(X)$, then for every $\lambda\in \mathcal{R}(T,X)$ concentrated on $U$ there exists $f \in \M(T,X)$ with $f(t)\in U(t)$ a.e.\ $t\in T$ such that
\begin{equation}
\label{pp1}
\int_T \int_X\Phi(t,x)\lambda(t,dx)d\mu=\int_T \Phi(t,f(t))d\mu. \tag{PP}
\end{equation}
\end{thm}

\begin{proof}
Define the multifunction $\Gamma:T\twoheadrightarrow E^*$ by $\Gamma(t)=\overline{\mathrm{co}}^{\,\mathit{w}^*}\Phi(t,U(t))$. Then $\Gamma$ is integrably bounded and weakly$^*\!$ compact, convex-\hspace{0pt}valued because of the hypotheses on $\Phi$ and $U$. Since $\Phi(t,\cdot)$ is bounded and Borel measurable for every $t\in T$, it is also weakly$^*\!$ scalarly integrable as a function from $X$ to $E^*$. Thus, we can define the Gelfand integral of $\Phi(t,\cdot)$ by $g_\lambda(t)=\int \Phi(t,x)\lambda(t,dx)$ for any $\lambda\in \mathcal{R}(T,X)$ concentrated on $U$. It is evident that the function $g_\lambda:T\to E^*$ is Gelfand integrable. Applying Lemma \ref{lem2} for every $t\in T$ to the probability space $(X,\mathrm{Borel}(X),\lambda(t))$, the weakly$^*\!$ closed convex set $\Gamma(t)\subset E^*$, and the Gelfand integrable function $\Phi(t,\cdot)$, we have $g_\lambda\in \mathcal{S}_\Gamma^1$. It follows from $\int\Gamma d\mu=\int \overline{\mathrm{co}}^{\,\mathit{w}^*}\Gamma d\mu$ (see \cite[Theorem 4]{po08} and \cite[Proposition 1]{sy08}) that there exists a Gelfand integrable selector $g$ from $\Gamma$ such that $\int gd\mu=\int g_\lambda d\mu$. In view of $g(t)\in \Phi(t,U(t))$ and Filippov's implicit function theorem (see \cite[Theorem III.38]{cv77}), there exists a measurable function $f:T\to X$ such that $g(t)=\Phi(t,f(t))$ and $f(t)\in U(t)$ a.e.\ $t\in T$. This $f$ is a desired control function.
\end{proof}

Theorem \ref{PP1} means that any ``relaxed'' control system $t\mapsto\hat{\Phi}(t,\lambda(t)):=\int\Phi(t,x)\lambda(t,dx)$   operated by $\lambda\in \mathcal{R}(T,X)$ consistent with the control set $U(t)$ is realized by adopting a ``purified'' control system $t\mapsto\Phi(t,f(t))$ operated by $f\in \M(T,X)$ with the feasibility constraint $f(t)\in U(t)$ in such a way that its Gelfand integral over $T$ is preserved with $\int\hat{\Phi}(t,\lambda(t))d\mu=\int\Phi(t,f(t))d\mu$. 

The converse of Theorem \ref{PP1} is as follows. 

\begin{thm}
\label{nec2}
Let $(T,\Sigma,\mu)$ be a finite measure space, $E$ be an infinite-dimensional separable Banach space, and $X$ be an uncountable compact Polish space. If for every integrably bounded measurable function $\Phi:T\times X\to E^*$ such that $\Phi(t,\cdot):X\to E^*$ is continuous in the weak$^*\!$ topology of $E^*$ for every $t\in T$ and for every $\lambda\in \mathcal{R}(T,X)$ there exists $f \in \M(T,X)$ satisfying \eqref{pp1}, then $(T,\Sigma,\mu)$ is saturated. 
\end{thm}

\begin{proof}
It follows from \cite[Theorem 3]{po09} that if $(T,\Sigma,\mu)$ is not saturated, then there exists an integrably bounded Carath\'eodory function $\varphi:T\times X\to \R$ and $\lambda\in \mathcal{R}(T,X)$ such that no $f\in \M(T,X)$ with $f(t)\in U(t)$ a.e.\ $t\in T$ satisfies
$$
\int_T \int_X\varphi(t,x)\lambda(t,dx)d\mu=\int_T \varphi(t,f(t))d\mu.
$$
Let $x^*\in E^*\setminus \{ 0 \}$ be given arbitrarily and define the function $\Phi:T\times X\to E^*$ by $\Phi(t,x)=\varphi(t,x)x^*$. Obviously, $\Phi$ is integrably bounded, (jointly) measurable, and $\Phi(t,\cdot)$ is continuous in the weak$^*\!$ topology of $E^*$ for every $t\in T$, but \eqref{pp1} is false. 
\end{proof}

\begin{rem}
For the case with $E=E^*=\R^n$, \eqref{pp1} holds under the nonatomicity hypothesis, which is a well-known result in control theory attributed to \cite[Theorem IV.3.14]{wa72}. The equivalence of saturation and the purification principle was established in \cite{ls09,po09} for the case where $\Phi$ takes values in the countable product $\R^\N$ of the real line ($\R^\N$ is a Fr\'echet space). In particular, the observation in \cite[Example 2.7]{ls06} that the purification principle fails without the use of nonatomic Loeb measures when $\Phi$ takes values in $\R^\N$ provides the basis for the necessity of saturation in \cite{ls09}. The equivalence of saturation and the purification principle for the case where $\Phi$ takes values in a separable Banach space is due to \cite{ks14b}. Theorems \ref{PP1} and \ref{nec2} are another characterization of saturation in terms of \eqref{pp1} in dual spaces of a separable Banach space. \end{rem}

\subsection{The Density Property}
Let $X$ be a Polish space. An extended real-\hspace{0pt}valued function $\varphi:T\times X\to \R\cup \{+\infty\}$ is called an \textit{integrand} if it is $\Sigma\otimes \mathrm{Borel}(X)$-\hspace{0pt}measurable. An integrand $\varphi$ is called a \textit{normal integrand} if $\varphi(t,\cdot)$ is lower semicontinuous on $X$ for every $t\in T$. A Carath\'eodory integrand is a normal integrand. Denote by $\C^1(T\times X,\mu)$ the space of integrably bounded Carath\'eodory integrands on $T\times X$. For each integrand $\varphi$, define the integral functional $J_\varphi:\mathcal{R}(T,X)\to \R\cup \{\pm \infty \}$ by $J_\varphi(\lambda)=\iint\varphi(t,x)\lambda(t,dx)d\mu$. The \textit{weak topology} on $\mathcal{R}(T,X)$ is defined as the coarsest topology for which every integral functionals $J_\varphi$, $\varphi\in \C^1(T\times X,\mu)$, are continuous. The weak topology of $\mathcal{R}(T,X)$ is also the coarsest topology for which $J_\varphi$ is lower semicontinuous for every nonnegative normal integrand $\varphi$ whenever $X$ is compact; see \cite[Lemma A.2]{ba84a}. If $T$ is a singleton, then the set $\mathcal{R}(T,X)$ coincides with the set $\Pi(X)$. In this case $\C^1(T\times X,\mu)$ coincides with the space $C_b(X)$ of bounded continuous functions on $X$ and the weak topology of $\mathcal{R}(T,X)$ is the topology of weak convergence of probability measures in $\Pi(X)$.

A sequence $\{ \lambda_i \}$ in $\mathcal{R}(T,X)$ \textit{converges weakly} to $\lambda$ if for every $\varphi\in \C^1(T\times X,\mu)$, we have $\lim_{i}J_\varphi(\lambda_i)=J_\varphi(\lambda)$. A sequence $\{ \lambda_i \}$ in $\mathcal{R}(T,X)$ \textit{converges narrowly} to $\lambda$ if for every $u\in C_b(X)$ and $A\in \Sigma$, we have
$$
\lim_{i\to \infty}\int_A\int_Xu(x)\lambda_i(t,dx)d\mu=\int_A\int_Xu(x)\lambda(t,dx)d\mu.
$$
It follows from the definitions that weak convergence implies narrow convergence in $\mathcal{R}(T,X)$. Furthermore, the converse is also true, i.e., weak and narrow convergence in $\mathcal{R}(T,X)$ are equivalent; see \cite[Theorem 4.10 and Remark 3.6]{ba00}. If $X$ is compact, then $\mathcal{R}(T,X)$ is compact and sequentially compact for the weak topology; see \cite[Lemma A.4]{ba84a}. 

\begin{thm}[density property]
\label{dens}
Let $(T,\Sigma,\mu)$ be a saturated finite measure space, $E$ be a separable Banach space, $C$ be a weakly$^*\!$ closed subset of $E^*$, and $X$ be a compact Polish space. Suppose that the following conditions are satisfied. 
\begin{enumerate}[\rm (i)]
\item $\Phi:T\times X\to E^*$ is an integrably bounded measurable function such that $\Phi(t,\cdot):X\to E^*$ is continuous in the weak$^*\!$ topology of $E^*$ for every $t\in T$;
\item $U:T\twoheadrightarrow X$ is a closed-\hspace{0pt}valued multifunction with $\mathrm{gph}\,U\in \Sigma\otimes \mathrm{Borel}(X)$;
\item $\left\{ \int_T\int_X\Phi(t,x)dPd\mu\mid P\in\Pi(X) \right\}\cap C$ is nonempty. 
\end{enumerate}
Further, define the subset of $\mathcal{R}(T,X)$ by
$$
\K:=\left\{ \lambda\in \mathcal{R}(T,X)\,\Bigl| \begin{array}{l} \int_T\int_X\Phi(t,x)\lambda(t,dx)d\mu\in C \\ \lambda(t)(U(t))=1\ \text{a.e.\ $t\in T$} \end{array} \right\}.
$$
We then have the following equality
$$
\K=\overline{\left\{ \delta_{f(\cdot)}\in \mathcal{R}(T,X)\,\Bigl|
\begin{array}{l}
\int_T \Phi(t,f(t))d\mu\in C,\,f\in \M(T,X) \\
f(t)\in U(t) \text{ a.e.\ $t\in T$}
\end{array}
\right\}}^{\,\mathit{w}}
$$
where $\overline{\{ \cdots \}}^{\,\mathit{w}}$ signifies the closure of the set with respect to the weak topology of $\mathcal{R}(T,X)$. 
\end{thm}

\begin{proof}
Let $\lambda_0\in \K$ be arbitrary and $\U_0$ be its arbitrary neighborhood. By definition of the weak topology of $\mathcal{R}(T,X)$, there exist $\varphi_1,\dots,\varphi_k$ in $\C^1(T\times X,\mu)$ with $k\in \mathbb{N}$ such that $|J_{\varphi_i}(\lambda)-J_{\varphi_i}(\lambda_0)|<1$, $i=1,\dots,k$ implies $\lambda\in \U_0$. Define $\Psi:T\times X\to E^*\times \R^k$ by $\Psi=(\Phi,\varphi_1,\dots,\varphi_k)$. Then $\Psi$ is an integrably bounded measurable function such that $\Psi(t,\cdot):X\to E^*\times \R^k$ is continuous in the weak$^*\!$ topology of $E^*\times \R^k$ for every $t\in T$. Define the subset $D$ of $E^*\times \R^k$ by 
$$
D=C\times \left\{ \left(\int_T\int_X\varphi_1(t,x)\lambda_0(t,dx)d\mu,\dots,\int_T\int_X\varphi_k(t,x)\lambda_0(t,dx)d\mu \right) \right\}. 
$$
Applying Theorem \ref{PP1} to the pair $(\Psi,U)$ yields the existence of $f\in \M(T,X)$ with $f(t)\in U(t)$ a.e.\ $t\in T$ such that $\int\Psi(t,f(t))d\mu=\iint\Psi(t,x)\lambda_0(t,dx)d\mu$. This implies that $J_{\varphi_i}(\delta_{f(\cdot)})=J_{\varphi_i}(\lambda_0)$ for $i=1,\dots,k$ and $\int\Phi(t,f(t))d\mu=\iint\Phi(t,x)\lambda_0(t,dx)d\mu\in C$. Therefore, $\delta_{f(\cdot)}\in \U_0$. Since the choice of $\lambda_0$ and $\U_0$ is arbitrary, this means the density property is as claimed. 
\end{proof}

\begin{rem}
\label{rem}
The case with $\Phi(t,x)\equiv 0$ and $C\equiv E^*$ means that the constraint in the dual space does not arise in control systems, in which case, the classical Lyapunov convexity theorem is sufficient for the density property and Theorem \ref{dens} is true even if $(T,\Sigma,\mu)$ is not saturated, but it is nonatomic; see \cite[Theorem IV.3.10]{wa72}. Moreover, if $U(t)\equiv X$, then every constraint is unbinding and Theorem \ref{dens} is reduced to the well-known result $\mathcal{R}(T,X)=\overline{\M(T,X)}^{\,\mathit{w}}$; see \cite[Theorem IV.2.6]{wa72}. For the density property with finite-dimensional control systems, see, e.g., \cite[Proposition II.7]{bl73} and \cite[Theorem 7 and Corollary 4]{sb78}.  
\end{rem}

\section{Variational Problems with Gelfand Integral Constraints}
\subsection{The Minimization Principle}
The variational problem under investigation is a general form of the isometric problems, which is an infinite-dimensional analogue of \cite{ap65} with the finite-dimensional setting. The relaxation technique explored here is a Gelfand integral analogue of \cite{ks14b} with the Bochner integral setting. For the existence issue in relaxation and purification in finite-dimensional control systems with integral constraints, see, e.g., \cite{al72,ar98,bl73}.

Suppose that an integrand $\varphi:T\times X\to \R\cup \{+\infty\}$ denotes a cost function and a constraint is described by a measurable function $\Phi:T\times X\to E^*$, a multifunction $U:T\twoheadrightarrow X$ and a given subset $C$ of $E^*$. The variational problem under consideration is
\begin{equation}
\label{vp}
\begin{aligned}
& \min_{f\in \M(T,X)} \int_T\varphi(t,f(t))d\mu \quad  \\
& \text{s.t. $\int_T \Phi(t,f(t))d\mu\in C$ and $f(t)\in U(t)$ a.e.\ $t\in T$}.
\end{aligned}
\tag{\text{VP}}
\end{equation}
Denote by $\min\eqref{vp}$ the minimum value of \eqref{vp} if it exists. The relaxed variational problem corresponding to \eqref{vp} is as follows.
\begin{equation}
\label{rvp}
\begin{aligned}
& \min_{\lambda\in \mathcal{R}(T,X)}\int_T\int_X\varphi(t,x)\lambda(t,dx)d\mu \quad  \\
& \text{s.t. $\int_T\int_X\Phi(t,x)\lambda(t,dx)d\mu\in C$ and $\lambda(t)(U(t))=1$ a.e. $t\in T$}.
\end{aligned}
\tag{\text{RVP}}
\end{equation}
Denote by $\min\eqref{rvp}$ the minimum value of \eqref{rvp} if it exists. Since any $f\in \M(T,X)$ is identified with $\delta_{f(\cdot)}\in \mathcal{R}(T,X)$ such that $\int\varphi(t,x)d(\delta_{f(t)})=\varphi(t,f(t))$ and $\int\Phi(t,x)d(\delta_{f(t)})=\Phi(t,f(t))$ for every $t\in T$, and the transformations $\lambda\mapsto (\int\varphi(t,x)\lambda(t,dx),\int\Phi(t,x)\lambda(t,dx))$ are affine on $\Pi(X)$ in their own right, \eqref{rvp} is a convexified problem to \eqref{vp} with $\min\eqref{vp}\ge \min\eqref{rvp}$ whenever solutions to both problems exist. (If the infimum value of \eqref{vp} happens to be $+\infty$, any feasible solutions to \eqref{vp} and \eqref{rvp} are optimal. Thus, we may innocuously assume that the infimum value of $\eqref{rvp}$ is less than $+\infty$.) 

\begin{lem}
\label{lem3}
Let $(T,\Sigma,\mu)$ be a finite measure space, $E$ be a Banach space, and $X$ be a Polish space. If $\Phi:T\times X\to E^*$ is an integrably bounded measurable function such that $\Phi(t,\cdot):X\to E^*$ is continuous in the weak$^*\!$ topology of $E^*$ for every $t\in T$, then the Gelfand integral functional $I_\Phi:\mathcal{R}(T,X)\to E^*$ defined by 
$$
I_\Phi(\lambda)=\int_T\int_X\Phi(t,x)\lambda(t,dx)d\mu
$$ 
is sequentially continuous in the weak topology of $\mathcal{R}(T,X)$ and the weak$^*\!$ topology of $E^*$.
\end{lem}

\begin{proof}
Let $\{ \lambda_i \}$ be a sequence in $\mathcal{R}(T,X)$ converging weakly to $\lambda$. Take $y\in E$ arbitrarily. We then have
\begin{align*}
\lim_{i\to \infty}\left\langle \int_T\int_X\Phi(t,x)\lambda_i(t,dx)d\mu,y \right\rangle
& =\lim_{i\to \infty}\int_T\left\langle \int_X\Phi(t,x)\lambda_i(t,dx),y \right\rangle d\mu \\
& =\lim_{i\to \infty}\int_T\left[ \int_X\langle \Phi(t,x),y \rangle \lambda_i(t,dx) \right]d\mu \\
& =\int_T\left[ \int_X\langle \Phi(t,x),y \rangle \lambda(t,dx) \right]d\mu \\
& =\left\langle \int_T\int_X\Phi(t,x)\lambda(t,dx)d\mu,y \right\rangle,
\end{align*}
where the third equality follows from the fact that the function $(t,x)\mapsto \langle \Phi(t,x),y \rangle$ belongs to $\C^1(T\times X,\mu)$ in view of $\int\langle \Phi(t,x),y \rangle \lambda_i(t,dx)\le \|y\| \psi(t)$ with $\psi\in L^1(\mu)$ for every $i$ and $t\in T$ by the integrable boundedness of $\Phi$ and the definition of the weak convergence in $\mathcal{R}(T,X)$. Therefore, $I_\Phi(\lambda_i)$ converges weakly$^*\!$ to $I_\Phi(\lambda)$ in $E^*$.
\end{proof}

\begin{thm}
\label{exst1}
Let $(T,\Sigma,\mu)$ be a finite measure space, $E$ be a separable Banach space, $C$ be a weakly$^*\!$ closed subset of $E^*$, and $X$ be a compact Polish space. Suppose that the following conditions are satisfied. 
\begin{enumerate}[\rm (i)]
\item $\varphi:T\times X\to \R\cup \{+\infty\}$ is a normal integrand such that there exists $\psi\in L^1(\mu)$ such that $\psi(t)\le \varphi(t,x)$ for every $(t,x)\in T\times X$;
\item $\Phi:T\times X\to E^*$ is an integrably bounded measurable function such that $\Phi(t,\cdot):X\to E^*$ is continuous in the weak$^*\!$ topology of $E^*$ for every $t\in T$; 
\item $U:T\twoheadrightarrow X$ is a closed-\hspace{0pt}valued multifunction with $\mathrm{gph}\,U\in \Sigma\otimes \mathrm{Borel}(X)$;
\item $\left\{ \int_T\int_X\Phi(t,x)dPd\mu\mid P\in\Pi(X) \right\}\cap C$ is nonempty.
\end{enumerate}
Then a solution to \eqref{rvp} exists. 
\end{thm}

\begin{proof}
Let $\{ \lambda_i \}$ be a minimizing sequence in $\mathcal{R}(T,X)$ for \eqref{rvp}. By the weak compactness of $\mathcal{R}(T,X)$, we can extract a subsequence from $\{ \lambda_i \}$ (which we do not relabel) that converges weakly to some $\lambda\in \mathcal{R}(T,X)$. Since $\{ \lambda_i \}$ converges narrowly to $\lambda\in \mathcal{R}(T,X)$ and each $\lambda_i$ is concentrated on $U$, we conclude that $\lambda$ is concentrated on $U$ as well; see \cite[Lemma 4.11 and Theorem 4.15]{ba00}. It follows from $\iint\Phi(t,x)\lambda_i(t,dx)d\mu\in C$ for each $i$ that $\iint\Phi(t,x)\lambda(t,dx)d\mu\in C$ by Lemma \ref{lem3} and the weak$^*\!$ closedness of $C$. Since $\varphi$ is integrably bounded from below, without loss of generality we may assume for the sake of \eqref{rvp} that $\varphi$ is a nonnegative normal integrand. Since the integral functional $J_\varphi(\nu)=\iint\varphi(t,x)\nu(t,dx)d\mu$ is weakly lower semicontinuous by the definition of the weak topology of $\mathcal{R}(T,X)$, we have $J_\varphi(\lambda)\le \liminf_i J_\varphi(\lambda_i)=\min\text{\eqref{rvp}}$. Hence, $\lambda$ is a solution to \eqref{rvp}. 
\end{proof}

For the existence of solutions to \eqref{rvp}, the saturation assumption on the measure space is unnecessary. Thus, the Lebesgue unit interval, the most fundamental probability space in many applications, is covered in Theorem \ref{exst1}. To ensure that $\min\text{\eqref{vp}}=\min\text{\eqref{rvp}}$ and the existence of solutions to \eqref{vp}, the saturation of the measure space and the continuity property on the normal integrand are sufficient.

\begin{cor}[minimization principle]
\label{exst2}
Let $(T,\Sigma,\mu)$ be a saturated finite measure space, $E$ be a separable Banach space, $C$ be a weakly$^*\!$ closed subset of $E^*$, and $X$ be a compact Polish space. Suppose that the following conditions are satisfied. 
\begin{enumerate}[\rm (i)]
\item $\varphi:T\times X\to \R$ is an integrably bounded Carath\'{e}odory function;
\item $\Phi:T\times X\to E^*$ is an integrably bounded measurable function such that $\Phi(t,\cdot):X\to E^*$ is continuous in the weak$^*\!$ topology of $E^*$ for every $t\in T$; 
\item $U:T\twoheadrightarrow X$ is a closed-\hspace{0pt}valued multifunction with $\mathrm{gph}\,U\in \Sigma\otimes \mathrm{Borel}(X)$;
\item $\left\{ \int_T\int_X\Phi(t,x)dPd\mu\mid P\in\Pi(X) \right\}\cap C$ is nonempty.
\end{enumerate}
Then a solution to \eqref{vp} exists.
\end{cor}

\begin{proof}
Let $\lambda\in \mathcal{R}(T,X)$ be a solution to \eqref{rvp}. Define the function $\Psi:T\times X\to E^*\times \R$ by $\Psi=(\Phi,\varphi)$ and the weakly$^*\!$ closed set $D\subset E^*\times \R$ by $D=C\times \{ \iint\varphi(t,x)\lambda(t,dx)d\mu \}$. Applying Theorem \ref{PP1} to the pair $(\Psi,U)$ yields the existence of $f\in \M(T,X)$ with $f(t)\in U(t)$ a.e.\ $t\in T$ satisfying $\iint\Psi(t,x)\lambda(t,dx)d\mu=\int\Psi(t,f(t))d\mu$. This means that $\int\varphi(t,f(t))d\mu=\iint\varphi(t,x)\lambda(t,dx)d\mu$ and $\int\Phi(t,f(t))d\mu=\iint\Phi(t,x)\lambda(t,dx)d\mu\in C$. Therefore, $\min\eqref{vp}=\min\eqref{rvp}$ and $f$ is a solution to \eqref{vp}. 
\end{proof}

We say that a quartet $(\varphi,\Phi,U,C)$ fulfills the \textit{minimization principle} (MP) if \eqref{vp} corresponding to $(\varphi,\Phi,U,C)$ has a solution. The converse of Corollary \ref{exst2} is framed in the following form. 

\begin{thm}
Let $(T,\Sigma,\mu)$ be a nonatomic finite measure space, $E$ be an infinite-\hspace{0pt}dimensional separable Banach space, $C$ be a weakly$^*\!$ closed subset of $E^*$, and $X$ be an uncountable compact Polish space. If every quartet $(\varphi,\Phi,U,C)$ satisfying conditions {\rm (i)} to {\rm (iv)} of Corollary \ref{exst2} fulfills {\rm (MP)}, then $(T,\Sigma,\mu)$ is saturated. 
\end{thm}
 
\begin{proof}
Suppose that the nonatomic finite measure space $(T,\Sigma,\mu)$ is not saturated. By Theorem \ref{nec2}, letting $U(t)\equiv X$ and $C\equiv E^*$ guarantees that there exists a quartet $(\varphi,\Phi,U,C)$ satisfying conditions (i) and (iv) of Corollary \ref{exst2} such that for some $\lambda\in \mathcal{R}(T,X)$ no $f\in \M(T,X)$ satisfies \eqref{pp1} and $J_\varphi(\lambda)=\int\varphi(t,f(t))d\mu$. Hence, the variational problem corresponding to $(\varphi,\Phi,U,C)$ yields the ``relaxation gap'': $\min\eqref{rvp}<\min\eqref{vp}$. Since $\mathcal{R}(T,X)=\overline{\M(T,X)}^{\,\mathit{w}}$ in view of the nonatomicity hypothesis (see Remark \ref{rem}) and $J_\varphi$ is weakly lower semicontinuous on $\mathcal{R}(T,X)$, there exists a minimizing sequence $\{ f_n \}$ in $\M(T,X)$ such that $J_\varphi(\delta_{f_n(\cdot)})\to \min\text{\eqref{rvp}}$. This means that $J_\varphi(\delta_{f_n(\cdot)})=\int\varphi(t,f_n(t))d\mu<\min\eqref{vp}$ for every sufficiently large $n$, an obvious contradiction. 
\end{proof}

The proof of Corollary \ref{exst2} is based on the ``direct method'' of the calculus of variations via the relaxation technique. For the existence result without the relaxation technique, based on the ``indirect method'' exploiting the duality theory in Asplund spaces in the nonsmooth setting, see \cite{sa15}. As investigated thoroughly in \cite{mo06}, Asplund spaces are suitable places for exploring subdifferential calculus fully. Another relevant application of saturation to subdifferential calculus in Asplund spaces for integral functionals is found in \cite{ms18}.

\section{Relaxation of Large Economies}
\subsection{Existence of Pareto Optimal Allocations}
Large economies introduced in \cite{au64,au66} model the set of agents as a nonatomic finite measure space in the finite-\hspace{0pt}dimensional framework to show the existence of Walrasian equilibria and the core-Walras equivalence without any convexity hypothesis; see also \cite{hi74} for detailed references to follow-up work till 1974. We apply the relaxation technique to exchange economies with an infinite-dimensional commodity space and make use of the purification principle to show the existence of Pareto optimal allocations for the original economy.  

The set of agents is given by a (complete) finite measure space $(T,\Sigma,\mu)$. The commodity space is given by the dual space $E^*$ of a separable Banach space $E$. The preference relation ${\succsim}(t)$ of each agent $t\in T$ is a complete, transitive binary relation on a common consumption set $X\subset E^*$, which induces the preference map $t\mapsto {\succsim}(t)\subset X\times X$. We denote by $x\,{\succsim}(t)\,y$ the relation $(x,y)\in {\succsim}(t)$. The indifference and strict relations are defined respectively by $x\,{\sim}(t)\,y$ $\Leftrightarrow$ $x\,{\succsim}(t)\,y$ and $y\,{\succsim}(t)\,x$, and by $x\,{\succ}(t)\,y$ $\Leftrightarrow$ $x\,{\succsim}(t)\,y$ and $x\,{\not\sim}(t)\,y$. Each agent possesses an initial endowment $\omega(t)\in X$, which is the value of a Gelfand integrable function $\omega:T\to E^*$. The economy $\E$ consists of the primitive $\E=\{ (T,\Sigma,\mu),X,\succsim,\omega \}$. 

The standing assumption on $\E$ is described as follows. 

\begin{assmp}
\label{assmp}
\begin{enumerate}[(i)]
\item $X$ is a weakly$^*\!$ compact subset of $E^*$. 
\item ${\succsim}(t)$ is a weakly$^*\!$ closed subset of $X\times X$ for every $t\in T$. 
\item For every $x,y\in X$ the set $\{ t\in T\mid x\,{\succ}(t)\,y \}$ is in $\Sigma$.  
\end{enumerate}
\end{assmp}

The preference relation ${\succsim}(t)$ is said to be \textit{continuous} if it satisfies Assumption \ref{assmp}(ii). Since $E$ is separable, the weakly$^*\!$ compact set $X\subset E^*$ is metrizable for the weak$^*\!$ topology (see \cite[Corollary 2.6.20]{me98}), and hence, the common consumption set $X$ is a compact Polish space. It follows from \cite[Proposition 1]{au69} that there exists a Carath\'eodory function $\varphi:T\times X\to \R$ such that  
\begin{equation}
\label{rp1}
\forall x,y\in X\ \forall t\in T: x\,{\succsim}(t)\,y \Longleftrightarrow \varphi(t,x)\ge \varphi(t,y). 
\end{equation}
(While \cite{au69} treated the case where $X$ is the nonnegative orthant of a finite-dimensional Euclidean space, the  proof is obviously valid as it stands for the case where $X$ is a separable metric space.) Moreover, this representation in terms of Carath\'eodory functions is unique up to strictly increasing, continuous transformations in the following sense: If $F:T\times \R\to \R$ is a function such that $t\mapsto F(t,r)$ is measurable and $r\mapsto F(t,r)$ is strictly increasing and continuous, then $x\,{\succsim}(t)\,y \Leftrightarrow F(t,\varphi(t,x))\ge F(t,\varphi(t,y))$, where $(t,x)\mapsto F(t,\varphi(t,x))$ is a Carath\'eodory function. In the sequel, we may assume without loss of generality that the preference map $t\mapsto {\succsim}(t)$ is represented by a Carath\'eodory function $\varphi$ that is unique up to strictly increasing, continuous transformations. 

Given a continuous preference ${\succsim}(t)$ on $X$, its continuous affine extension ${\succsim}_\mathcal{R}(t)$ to $\Pi(X)$ is obtained by convexifying (randomizing) the individual utility function $\varphi(t,\cdot)$ in such a way  
\begin{equation}
\label{rp2}
\forall P,Q\in \Pi(X)\ \forall t\in T: P\,{\succsim}_\mathcal{R}(t)\,Q \stackrel{\text{def}}{\Longleftrightarrow}  \int_X \varphi(t,x)dP\ge \int_X \varphi(t,x)dQ.
\end{equation}
The continuous extension ${\succsim}_{\mathcal{R}}(t)$ of ${\succsim}(t)$ from $X$ to the \textit{relaxed consumption set} $\Pi(X)$ is called a \textit{relaxed preference relation} on $\Pi(X)$. Thus, the restriction of ${\succsim_\mathcal{R}}(t)$ to $\Delta(X)$ coincides with ${\succsim}(t)$ on $X$. Indifference relation ${\sim}_\mathcal{R}(t)$ and strict relation ${\succ}_\mathcal{R}(t)$ are defined in a way analogous to the above. The extension formula in \eqref{rp2} conforms to the relaxation technique investigated in Section 3. As observed in \cite{ks16b}, relaxed preferences are also consistent with the axioms for the ``expected utility hypothesis'' and the continuous function $\varphi(t,\cdot)$ corresponds to the ``von Neumann--Morgenstern utility function'' for ${\succsim}_\mathcal{R}(t)$. 

Denote by $\E_\mathcal{R}=\{ (T,\Sigma,\mu),\Pi(X),{\succsim}_\mathcal{R},\delta_{\omega(\cdot)} \}$ the  \textit{relaxed economy} induced by the original economy $\E=\{ (T,\Sigma,\mu),X,\succsim,\omega \}$, where the initial endowment $\omega(t)\in X$ of each agent is identified with a Dirac measure $\delta_{\omega(t)}\in \Delta(X)$, and hence, $\delta_{\omega(\cdot)}\in \mathcal{R}(T,X)$. Let $\imath_X$ be the identity map on $X$. We denote by $\int\imath_XdP$ the Gelfand integral of $\imath_X$ with respect to the probability measure $P\in \Pi(X)$.  

To deal with Pareto optimality with or without free disposal simultaneously, following \cite[Chapter 8]{mo06}, we introduce ``market constraints'' for the definition of (relaxed) allocations.  

\begin{dfn}
Let $W$ be a nonempty subset of $E^*$. 
\begin{enumerate}[(i)]
\item A Gelfand integrable function $f\in G^1(\mu,E^*)$ is an \textit{allocation} for $\E$ if it satisfies:
$$
\int_Tf(t)d\mu-\int_T\omega(t)d\mu\in W \quad\text{and $f(t)\in X$ a.e. $t\in T$}.
$$
\item A relaxed control $\lambda\in \mathcal{R}(T,X)$ is a \textit{relaxed allocation} for $\E_\mathcal{R}$ if it satisfies: 
$$
\int_T\int_X\imath_X(x)\lambda(t,dx)d\mu-\int_T\omega(t)d\mu\in W. 
$$   
\end{enumerate}
\end{dfn}
\noindent
In particular, when $W=\{ 0 \}$, the definition reduces to the (relaxed) allocations ``without'' \textit{free disposal}; when $-W$ is a convex cone and $E$ is endowed with the cone order $\le$ defined by $x\le y \Leftrightarrow y-x\in -W$, the definition reduces to the (relaxed) allocations ``with'' free disposal. Denote by $\A(\E)$ the set of allocations for $\E$ and by $\A(\E_\mathcal{R})$ the set of relaxed allocations for $\E_\mathcal{R}$. If $\lambda$ is a relaxed allocation for $\E_\mathcal{R}$ such that $\lambda(t)=\delta_{f(t)}\in \Delta(X)$ for every $t\in T$ and $f\in G^1(\mu,E^*)$, then it reduces to the usual feasibility constraint $\int fd\mu-\int \omega d\mu\in W$ for $\E$. This means that $\A(\E)\subset \A(\E_\mathcal{R})$. 

An immediate consequence of Theorem \ref{dens} leads to the density property of the set of allocations.

\begin{thm}
Let $(T,\Sigma,\mu)$ be a saturated finite measure space, $E$ be a separable Banach space, $X$ be a weakly$^*\!$ compact subset of $E^*$, and $W$ be a weakly$^*\!$ closed subset of $E^*$. Then $\A(\E_\mathcal{R})=\overline{\A(\E)}^{\mathit{\,w}}$.
\end{thm}

\begin{proof}
Simply apply Theorem \ref{dens} to the case with $\Phi(t,x)\equiv \imath_X(x)$, $U(t)\equiv X$, and $C=\int\omega d\mu+W$. 
\end{proof}

\begin{dfn}
\begin{enumerate}[(i)]
\item An allocation $f\in \A(\E)$ is \textit{Pareto optimal} for $\E$ if there exist no $g\in \A(\E)$ and $A\in \Sigma$ of positive measure such that $g(t)\,{\succsim}(t)\,f(t)$ a.e.\ $t\in T$ and $g(t)\,{\succ}(t)\,f(t)$ for every $t\in A$.
\item A relaxed allocation $\lambda\in \A(\E_{\mathcal{R}})$ is \textit{Pareto optimal} for $\E_\mathcal{R}$ if there exist no $\nu\in \mathcal{R}(T,X)$ and $A\in \Sigma$ of positive measure such that $\nu(t)\,{\succsim}_\mathcal{R}(t)\,\lambda(t)$ a.e.\ $t\in T$ and $\nu(t)\,{\succ}_\mathcal{R}(t)\,\lambda(t)$ for every $t\in A$.
\end{enumerate}
\end{dfn}
\noindent
Denote by $\P(\E)$ the set of Pareto optimal allocations for $\E$ and by $\P(\E_\mathcal{R})$  the set of Pareto optimal relaxed allocations for $\E_\mathcal{R}$. 

\begin{thm}
\label{exst3}
Let $(T,\Sigma,\mu)$ be a finite measure space, $E$ be a separable Banach space, and $W$ be a weakly$^*\!$ closed subset of $E^*$. Then $\P(\E_\mathcal{R})$ is nonempty for every economy $\E$ satisfying Assumption \ref{assmp}. 
\end{thm}

\begin{proof}
If the Carath\'{e}odory integrand $\varphi$ in the preference representation \eqref{rp1} happens to be integrably unbounded, then choose any Carath\'{e}odory function $F:T\times \R \to \R$ such that $F(t,\cdot)$ is strictly increasing for every $t\in T$ and there exists $\psi\in L^1(\mu)$ satisfying $|F(t,r)|\le \psi(t)$ for every $(t,r)\in T\times \R$, and consider the transformation $(t,x)\mapsto F(t,\varphi(t,x))$ of the preference representation, which is obviously an integrably bounded Carath\'{e}odory integrand preserving \eqref{rp1}. (For example, letting $\tilde{\varphi}(t):=\max_{x\in X}|\varphi(t,x)|$ and $F(t,r):=e^{-\tilde{\varphi}(t)}r$ yields $|F(t,\varphi(t,x))|\le 1$ for every $(t,x)\in T\times X$.) Thus, without loss of generality we may assume that $\varphi$ is integrably bounded. Consider \eqref{rvp} with $\Phi(t,x)\equiv \imath_X(x)$, $U(t)\equiv X$, and $C=\int\omega d\mu+W$, which is reduced to the variational problem with a Gelfand integral constraint 
\begin{equation}
\label{rvp2}
\begin{aligned}
& \max_{\lambda\in \mathcal{R}(T,X)}\int_T\int_X\varphi(t,x)\lambda(t,dx)d\mu \\
& \text{s.t. }\int_T\int_X\imath_X(x)\lambda(t,dx)d\mu-\int_T\omega(t)d\mu\in W.
\end{aligned}
\tag{RVP$'$}
\end{equation}
Suppose that the solution $\lambda$ to \eqref{rvp2} does not belong to $\P(\E_\mathcal{R})$. Then there exist $\nu\in \A(\E_\mathcal{R})$ and $A\in \Sigma$ of positive measure such that $\nu(t)\,{\succsim}_\mathcal{R}(t)\,\lambda(t)$ a.e.\ $t\in T$ and $\nu(t)\,{\succ}_\mathcal{R}(t)\,\lambda(t)$ for every $t\in A$. Given the preference formula \eqref{rp2}, this is equivalent to $\int\varphi(t,x)\nu(t,dx)\ge \int\varphi(t,x)\lambda(t,dx)$ a.e.\ $t\in T$ and $\int\varphi(t,x)\nu(t,dx)>\int\varphi(t,x)\lambda(t,dx)$ for every $t\in A$. Integrating these inequalities over $T$ and adding them up yield $\iint\varphi(t,x)\nu(t,dx)d\mu>\iint\varphi(t,x)\lambda(t,dx)d\mu$, a contradiction to the fact that $\lambda$ is a solution to \eqref{rvp2}.
\end{proof}

It should be noted that to guarantee that $\P(\E_\mathcal{R})$ is nonempty, the saturation and even the nonatomicity hypotheses are unnecessary as well as the convexity hypothesis. Under the saturation hypothesis, the existence of Pareto optimal allocations for the original economy is guaranteed. 

\begin{thm}
\label{exst4}
Let $(T,\Sigma,\mu)$ be a saturated finite measure space, $E$ be a separable Banach space, and $W$ be a weakly$^*\!$ closed subset of $E^*$. Then $\P(\E)$ is nonempty for every economy $\E$ satisfying Assumption \ref{assmp}. 
\end{thm}

\begin{proof}
As shown in the proof of Theorem \ref{exst3}, any solution $\lambda\in \mathcal{R}(T,X)$ to \eqref{vp2} belongs to $\P(\E_\mathcal{R})$. It follows from Proposition \ref{PP1} that there exists $f\in \A(\E)$ such that $\iint \varphi(t,x)\lambda(t,dx)=\int\varphi(t,f(t))d\mu$. If $f$ is not a solution to the variational problem
\begin{equation}
\label{vp2}
\begin{aligned}
& \max_{f\in \mathcal{M}(T,X)}\int_T\varphi(t,f(t))d\mu \\
& \text{s.t. }\int_Tf(t)d\mu-\int_T\omega(t)d\mu\in W
\end{aligned}
\tag{VP$'$}
\end{equation}
then there exists $g\in \A(\E)$ such that $\int\varphi(t,g(t))d\mu>\int\varphi(t,f(t))d\mu$. Since $\delta_{g(\cdot)}\in \A(\E_\mathcal{R})$, the above inequality obviously contradicts the fact that $\lambda\in \mathcal{R}(T,X)$ is a solution to \eqref{vp2}. Suppose that the solution $f$ to \eqref{vp2} does not belong to $\P(\E)$. Then there exist $g\in \A(\E)$ and $A\in \Sigma$ of positive measure such that $g(t)\,{\succsim}(t)\,f(t)$ a.e.\ $t\in T$ and $g(t)\,{\succ}(t)\,f(t)$ for every $t\in A$. Given the preference formula \eqref{rp1}, this is equivalent to $\varphi(t,g(t))\ge \varphi(t,f(t))$ a.e.\ $t\in T$ and $\varphi(t,g(t))>\varphi(t,f(t))$ for every $t\in A$. Integrating these inequalities over $T$ and adding them up yield $\int\varphi(t,g(t))d\mu>\int\varphi(t,f(t))d\mu$, a contradiction to the fact that $f$ is a solution to \eqref{vp2}.
\end{proof}

The existence of (relaxed) Walrasian equilibria with free disposal is investigated in \cite{ks16b} for the commodity space with the dual space of $L^\infty=(L^1)^*$. The crucial argument for the proof is the nonemptiness of the norm interior of the positive cone of $L^\infty$, which fails to hold for general dual spaces. It is a challenging open question to establish the existence of (relaxed) Walrasian equilibria for general dual spaces. 

\begin{rem}
The role of the weak$^*\!$ compactness of the consumption set $X$ in Assumption \ref{assmp} is twofold. The first role is to guarantee the existence of continuous utility functions. To apply the celebrated Debreu's utility representation theorem, $X$ is required to satisfy the second axiom of countability; in particular, it needs to be a separable metric space. Without the weak$^*\!$ compactness assumption, $X$ may not be a separable metric space with respect to the weak$^*\!$ topology even if $E$ is separable, which prevents one from obtaining continuous utility functions representing continuous preference relations. The second role is to guarantee the existence of solutions to \eqref{rvp} in Theorem \ref{exst1}. The lack of compactness of $X$ inevitably leads to the noncompactness of $\mathcal{R}(T,X)$ and $\Pi(X)$, and hence, to a possible failure of Theorems \ref{exst3} and \ref{exst4}.  
\end{rem}

\end{document}